\documentclass[11pt,reqno,a4paper]{amsart}

\usepackage[T1]{fontenc}
\usepackage{lmodern}
\usepackage[margin=1.08in]{geometry}
\usepackage{amsmath,amssymb,amsthm,mathtools}
\usepackage{microtype}
\usepackage[numbers,sort&compress]{natbib}
\usepackage{xcolor}
\usepackage[colorlinks=true,
  linkcolor=blue!55!black,
  citecolor=blue!55!black,
  urlcolor=blue!55!black]{hyperref}

\numberwithin{equation}{section}

\newtheorem{theorem}{Theorem}[section]
\newtheorem{proposition}[theorem]{Proposition}
\newtheorem{lemma}[theorem]{Lemma}
\newtheorem{corollary}[theorem]{Corollary}
\theoremstyle{definition}

\theoremstyle{remark}
\newtheorem{remark}[theorem]{Remark}

\newcommand{\R}{\mathbb R}
\newcommand{\Z}{\mathbb Z}
\newcommand{\T}{\mathbb T}
\newcommand{\cC}{\mathcal C}
\newcommand{\cD}{\mathcal D}
\newcommand{\cE}{\mathcal E}
\newcommand{\cF}{\mathcal F}
\newcommand{\cL}{\mathcal L}
\newcommand{\cM}{\mathcal M}
\newcommand{\cP}{\mathcal P}
\newcommand{\cT}{\mathcal T}
\newcommand{\cZ}{\mathcal Z}
\newcommand{\transpose}{\mathsf T}
\DeclareMathOperator{\arsinh}{arsinh}
\DeclareMathOperator{\diag}{diag}
\newcommand{\divd}{\operatorname{div}_{\!d}}

\title[Principal eigenvalues on periodic lattices]
{Variational Principles and Rearrangement Inequalities for asymmetric Operators on Periodic Lattices}

\author{Xing Liang}\thanks{School of Mathematical Sciences, University of Science and Technology of China, Hefei, Anhui 230026,  China  (\texttt{xliang@ustc.edu.cn}).  The author is supported by the National Natural Science Foundation of China (12331006,12531008) and XDB0900100}

\date{}


\begin{document}

\begin{abstract}
In this paper, we establish variational formulas and a rearrangement inequality for
principal eigenvalues of asymmetric second-order difference operators on
periodic lattices.  For a general irreducible nearest-neighbor operator,
positive right and left eigenvectors give an explicit saddle point and hence
equal minimum--maximum and maximum--minimum formulas over positive profiles
and periodic logarithmic correctors.  The corrector is unique and satisfies
a nonlinear discrete flux-conservation law.  For exponentially tilted
symmetric diffusion, the formula separates the discrete Dirichlet and
potential terms from an exact nonlinear periodic correction.  In one
dimension the flux is constant, and an elementary scalar-flux representation
yields a bell-shaped cyclic rearrangement that maximizes the tilted
principal eigenvalue for every tilt.
\end{abstract}

\maketitle

\section{Introduction}

\subsection{Principal eigenvalues beyond Rayleigh--Ritz}

Variational formulas for principal eigenvalues of elliptic or difference operators are among the most useful
links between spectral theory, energy methods, and optimization.  In the
self-adjoint case, we usually have the Rayleigh--Ritz principle of the principal eigenvalues and then maximize the principal eigenvalues via potential rearrangement 
based on the P\'olya--Szeg\H{o} inequality and Hardy--Littlewood inequality.  On the other hand, when
the operator is asymmetric, the principal eigenvalue remains real and
distinguished by positivity, but the ordinary quadratic form no longer
characterizes it and then the rearrangement discussion is not direct anymore.

This paper studies the variational representation and rearrangement inequalities for the 
principal eigenvalue of an asymmetric difference operator on periodic
lattices:  Let \(m\geq1\), 
\(\mathbf N=(N_1,\ldots,N_m)\) with \(N_\nu\geq3\), and write
\[
  \T_{\mathbf N}:=\prod_{\nu=1}^m\Z/N_\nu\Z .
\]
All coefficients are extended periodically to \(\Z^m\).  Given a real
diagonal coefficient \(q_x\) and positive periodic nearest-neighbor coefficients
\(r_\nu^\pm(x)\), consider
\begin{equation}\label{eq:operator}
  (\cM u)(x)
  =q_xu(x)+\sum_{\nu=1}^m
   \left[r_\nu^+(x)u(x+e_\nu)
        +r_\nu^-(x)u(x-e_\nu)\right].
\end{equation}
Its matrix is irreducible and Metzler, meaning that all off-diagonal
entries are nonnegative.  Hence it has a real, algebraically simple
principal eigenvalue \(\Lambda(\cM)\), equal to the spectral bound, and
positive right and left eigenvectors.

We ask two related questions.  First, is there a Rayleigh-type formula for
\(\Lambda(\cM)\) that displays the directed edge structure of \(\cM\)?
Second, for an exponentially tilted discrete diffusion operator with non-constant potential on a one-dimensional periodic
cell, how to rearrange the potential for maximizing
the principal eigenvalue?  

The Donsker--Varadhan theory and its finite-dimensional formulations provide
general variational principles for positive operators
\cite{DonskerVaradhan1975,DonskerVaradhan1976,Friedland1981}.  Our purpose is
more geometric.  We seek a representation in which the right and left
eigenvectors determine a positive profile and a periodic logarithmic
corrector, the asymmetric nearest-neighbor coefficients appear edge by edge,
and the Euler equation becomes a discrete flux-conservation law.  The
left--right eigenvector pair will in fact produce an explicit saddle point,
so both extremal orders follow without an abstract minimax theorem.  This
form reveals an exact nonlinear periodic corrector problem and, in one
dimension, a four-point inequality strong enough to solve the cyclic
rearrangement problem.

\subsection{The continuous case: Nadin's work}

The closest continuous precedent is Nadin's variational characterization
for periodic asymmetric elliptic operators \cite{Nadin2010}.  We describe
it here as a spectral and rearrangement result; its original motivation from
KPP propagation will be recalled below.  Let \(C\) be a periodicity cell,
let \(A(x)\) be periodic, symmetric, and uniformly elliptic, let \(b(x)\) be
a periodic vector field, and let \(\mu(x)\) be a periodic potential.  For
\(e\in\mathbb S^{m-1}\) and \(\lambda\in\R\),  \cite{Nadin2010} considered
\[
\begin{aligned}
 -L_{\lambda,e}\phi
 ={}&-\nabla\!\cdot(A(x)\nabla\phi)
     -2\lambda\,e\!\cdot A(x)\nabla\phi-b(x)\!\cdot\nabla\phi\\
 &-\bigl(\lambda^2e\!\cdot A(x)e
   +\lambda\nabla\!\cdot(A(x)e)
   +\lambda b(x)\!\cdot e+\mu(x)\bigr)\phi .
\end{aligned}
\]
The principal eigenvalue \(k_{\lambda e}(A,b,\mu)\) is determined by a
positive periodic eigenfunction.  Thus the framework is broader than an
exponential twist of a symmetric diffusion operator: at \(\lambda=0\) it
already contains the general asymmetric operator
\[
 -L_0\phi=-\nabla\!\cdot(A\nabla\phi)-b\!\cdot\nabla\phi-\mu\phi.
\]
\cite{Nadin2010}  proved, in particular,
\[
 k_0(A,b,\mu)
 =\max_{\beta\in C^1_{\rm per}}
 k_0\!\left(A,0,
   \mu+\nabla\beta\!\cdot A\nabla\beta
   +b\!\cdot\nabla\beta-\frac12\nabla\!\cdot b\right).
\]
When \(b=0\), the tilted formula takes the effective-diffusivity form
\begin{equation}\label{eq:nadin}
  k_{\lambda e}(A,0,\mu)
  =\min_{\substack{\alpha\in H^1_{\rm per}(C),\ \alpha>0\\
                    \int_C\alpha^2\,dx=1}}
  \left\{
    \int_C\nabla\alpha\cdot A\nabla\alpha\,dx
    -\int_C\mu\alpha^2\,dx
    -\lambda^2|C|D_e(\alpha^2A)
  \right\},
\end{equation}
where the effective diffusivity of the weighted matrix field
\(\alpha^2A\) in the unit direction \(e\) is
\[
 D_e(\alpha^2A)
 :=\min_{\substack{\chi\in H^1_{\rm per}(C)\\
                    \int_C\chi\,dx=0}}
 \frac1{|C|}\int_C
 (e+\nabla\chi)\cdot\alpha^2A(x)(e+\nabla\chi)\,dx .
\]
Thus the subscript \(e\) records the direction, while \(\chi\) is the
periodic corrector; the zero-mean condition merely fixes its additive
constant.  Nadin's contribution is
therefore not merely another abstract min--max identity.  It converts a
asymmetric principal eigenvalue into a positive-profile--corrector problem
adapted to periodic homogenization and Schwarz rearrangement.  This is also
what makes the formula effective for simultaneous optimization of an entire
tilted eigenvalue family.

In the one-dimensional constant-diffusion case, write
\(k_\lambda(\mu):=k_\lambda(1,0,\mu)\), and let \(\mu^*\) denote the
periodic Schwarz rearrangement of \(\mu\), namely the
equimeasurable profile that is symmetric and nonincreasing away from one
center of the period cell.  Nadin's rearrangement theorem states
\cite{Nadin2010} that
\[
  k_\lambda(\mu^*)\leq k_\lambda(\mu)
  \qquad\text{for every }\lambda\in\R.
\]
Once \eqref{eq:nadin} is available, this comparison follows directly from
the classical rearrangement inequalities.  Indeed, if \(\alpha^*\) is the
periodic Schwarz rearrangement of an admissible profile \(\alpha\), then,
after an irrelevant additive shift of \(\mu\) if necessary,
\[
 \int_C |(\alpha^*)'|^2\leq\int_C|\alpha'|^2,
 \qquad
 \int_C\mu^*(\alpha^*)^2\geq\int_C\mu\alpha^2,
 \qquad
 \int_C(\alpha^*)^{-2}=\int_C\alpha^{-2}.
\]
The first two relations improve the Dirichlet and potential terms, whereas
the last one leaves the one-dimensional effective-diffusivity term
unchanged.  Taking the minimum over positive profiles gives the spectral
inequality.  Thus Nadin's formula not only characterizes the asymmetric
eigenvalue; it turns the subsequent rearrangement step into a direct
application of standard continuous symmetrization principles.

\subsection{Why the continuous calculation does not discretize directly}

The distinction is already visible in the one-dimensional conserved
quantity.  For the constant-diffusion specialization of Nadin's operator,
\[
 -L_\lambda\phi=-\phi''-2\lambda\phi'-(\mu+\lambda^2)\phi,
\]
let \(\phi\) and \(\psi\) be the positive periodic eigenfunctions of the
operator and its adjoint, respectively, and set
\[
 \alpha=\sqrt{\phi\psi},\qquad
 \beta=\frac12\log\frac{\phi}{\psi}.
\]
The transformed eigenvalue equations give
\[
 \bigl[\alpha^2(\beta'+\lambda)\bigr]'=0,
 \qquad \alpha^2(\beta'+\lambda)=J.
\]
Here \(J\) is independent of the spatial variable.
Consider the exponentially tilted discrete diffusion operator on a one-dimensional periodic
cell,
\begin{equation}\label{oneDdiffu}
 (\cL_{p,V}u)_j
 =d(e^{-p}u_{j+1}+e^pu_{j-1}-2u_j)+V_ju_j,
 \qquad j\in\Z/N\Z,\quad d>0,
\end{equation}
on the periodic lattice. Let
\(\phi=(\phi_j)\) and \(\psi=(\psi_j)\) be its positive right and left
eigenvectors, and set
\[
 a_j=\sqrt{\phi_j\psi_j},\qquad
 \beta_j=\frac12\log\frac{\phi_j}{\psi_j}.
\]
We found that the corresponding law is instead
\[
 a_ja_{j+1}\sinh(\beta_{j+1}-\beta_j-p)=J,
\]
where \(J\) is independent of \(j\).  Unlike the continuous identity, this
law is nonlinear and weighted by neighboring products.  Consequently, the
exact discrete corrector term is built from \(\cosh\), not from a quadratic
gradient, and Nadin's rearrangement argument cannot be transferred by merely
replacing derivatives with differences.  The detailed comparison is given
in the one-dimensional reduction below.

\subsection{Main results }

For the general operator \(\cM\) in \eqref{eq:operator} and the
positively oriented edge \((x,x+e_\nu)\), define
\[
 \kappa_{x,\nu}
 :=\sqrt{r_\nu^+(x)r_\nu^-(x+e_\nu)},\qquad
 \gamma_{x,\nu}
 :=\frac12\log\frac{r_\nu^+(x)}{r_\nu^-(x+e_\nu)} .
\]
Let
\[
 \cP_{\mathbf N}
 :=\left\{a:\T_{\mathbf N}\to(0,\infty):
            \sum_xa_x^2=1\right\}.
\]
We fix the additive constant of every discrete corrector by introducing the
mean-zero periodic space
\[
 \cZ_{\mathbf N}
 :=\left\{\beta:\T_{\mathbf N}\to\R:\sum_x\beta_x=0\right\}.
\]
For \(a\in\cP_{\mathbf N}\) and \(\beta\in\cZ_{\mathbf N}\), set
\[
\begin{aligned}
\cF_{\cM}(a,\beta)
 :=\sum_xq_xa_x^2
 +2\sum_{x,\nu}\kappa_{x,\nu}a_xa_{x+e_\nu}
 \cosh\!\bigl(
   \beta_{x+e_\nu}-\beta_x+\gamma_{x,\nu}
 \bigr).
\end{aligned}
\]
Let
\[
 D_\beta:=\diag(e^{\beta_x}),\qquad
 H_\beta:=\frac12\left(
 D_\beta^{-1}\cM D_\beta
 +D_\beta\cM^{\transpose}D_\beta^{-1}
 \right).
\]
Pairing the two orientations of every edge gives
\(\cF_{\cM}(a,\beta)=a^{\transpose}H_\beta a\).
Since \(H_\beta\) is symmetric and irreducible Metzler, its largest
eigenvalue has a positive eigenvector and hence
\(\lambda_{\max}(H_\beta)=\max_{a\in\cP_{\mathbf N}}
\cF_{\cM}(a,\beta)\).

\par\medskip
\noindent\begin{minipage}{\textwidth}
\emph{Notation.}
Throughout, \(a\), \(\beta\), and \(J\) denote the positive profile, periodic
logarithmic corrector, and conserved edge flux; \(\kappa\) and \(\gamma\)
denote the symmetric edge coefficient and the local edge asymmetry.  In the
diffusion specialization, \(c\), \(p\), and \(V\) denote the edge diffusion
coefficient, tilt, and potential.  Unless stated otherwise,
\(\sum_{x,\nu}\) runs over \(x\in\T_{\mathbf N}\) and
\(1\leq\nu\leq m\).
\end{minipage}
\par\medskip

\begin{theorem}[Discrete variational principle]
\label{thm:highdim}
For the periodic operator \eqref{eq:operator},
\begin{equation}\label{eq:highdim}
\begin{aligned}
  \Lambda(\cM)
  &=\min_{\beta\in\cZ_{\mathbf N}}\lambda_{\max}(H_\beta)\\
  &=\min_{\beta\in\cZ_{\mathbf N}}
    \max_{a\in\cP_{\mathbf N}}\cF_{\cM}(a,\beta)\\
  &=\max_{a\in\cP_{\mathbf N}}
    \min_{\beta\in\cZ_{\mathbf N}}\cF_{\cM}(a,\beta).
\end{aligned}
\end{equation}
\end{theorem}
\begin{remark}
The equality of the two extremal orders is realized by an explicit saddle
point, rather than obtained from a general minimax theorem.  The inner
minimizer is unique in \(\cZ_{\mathbf N}\).  If
\(\cM\phi=\Lambda(\cM)\phi\) and
\(\cM^{\transpose}\psi=\Lambda(\cM)\psi\), where the positive
eigenvectors are scaled so that
\(\sum_x\phi_x\psi_x=1\) and
\(\sum_x\log(\phi_x/\psi_x)=0\), then the optimizing
profile--corrector pair is
\[
  a_x=\sqrt{\phi_x\psi_x},\qquad
  \beta_x=\frac12\log\frac{\phi_x}{\psi_x}.
\]
At the inner minimum, the edge flux
\[
 J_{x,\nu}
 :=\kappa_{x,\nu}a_xa_{x+e_\nu}
   \sinh\!\bigl(
     \beta_{x+e_\nu}-\beta_x+\gamma_{x,\nu}
   \bigr)
\]
is divergence-free.
\end{remark}

For the multidimensional exponential tilt of \eqref{eq:operator}, set
\[
 r_{\nu,p}^+(x)=r_\nu^+(x)e^{-p_\nu},\qquad
 r_{\nu,p}^-(x)=r_\nu^-(x)e^{p_\nu}.
\]
The general diagonal and edge fields in Theorem~\ref{thm:highdim} are
\[
 q_x=V(x)-\sum_{\nu=1}^m
       \bigl(r_\nu^+(x)+r_\nu^-(x)\bigr),
 \qquad
 \gamma_{x,\nu}(p)=\gamma_{x,\nu}-p_\nu.
\]
Thus every tilted formula below is a specialization of the general
variational principle, rather than part of its hypotheses.

To specialize Theorem~\ref{thm:highdim} to a discrete diffusion matrix, let
\(c_{x,\nu}>0\) be periodic edge diffusion coefficients and consider the symmetric
periodic discrete diffusion operator
\[
 (\cD^c u)(x)
 :=\sum_\nu\Bigl[
 c_{x,\nu}\bigl(u(x+e_\nu)-u(x)\bigr)
 +c_{x-e_\nu,\nu}\bigl(u(x-e_\nu)-u(x)\bigr)
 \Bigr].
\]
Its exponentially tilted diffusion--growth operator is
\[
\begin{aligned}
  (\cL^{\,c}_{p,V}u)(x)
  =\sum_\nu\bigl[&
   c_{x,\nu}e^{-p_\nu}u(x+e_\nu)
  +c_{x-e_\nu,\nu}e^{p_\nu}u(x-e_\nu)\\
  &-\bigl(c_{x,\nu}+c_{x-e_\nu,\nu}\bigr)u(x)\bigr]+V(x)u(x).
\end{aligned}
\]
Define the nonlinear periodic correction
\[
 \cE_p^c(a)
 :=\min_{\beta\in\cZ_{\mathbf N}}
 2\sum_{x,\nu}c_{x,\nu}a_xa_{x+e_\nu}
 \left[
 \cosh(\beta_{x+e_\nu}-\beta_x-p_\nu)-1
 \right].
\]

\begin{corollary}[Periodic discrete diffusion]\label{cor:reciprocal}
The principal eigenvalue of \(\cL^{\,c}_{p,V}\) satisfies
\begin{equation}\label{eq:reciprocal}
 \Lambda_p^c(V)
 =\max_{a\in\cP_{\mathbf N}}
 \left\{
  -\sum_{x,\nu}c_{x,\nu}(a_{x+e_\nu}-a_x)^2
  +\sum_xV(x)a_x^2+\cE_p^c(a)
 \right\}.
\end{equation}
In particular, \(\cE_0^c=0\) and \eqref{eq:reciprocal} reduces to the
ordinary Rayleigh formula at zero tilt.
\end{corollary}

\emph{Comparison with Nadin's formula.}
Multiplying \eqref{eq:reciprocal} by \(-1\) gives
\[
 -\Lambda_p^c(V)=\min_{a\in\cP_{\mathbf N}}
 \left\{
  \sum_{x,\nu}c_{x,\nu}(a_{x+e_\nu}-a_x)^2
  -\sum_xV(x)a_x^2-\cE_p^c(a)
 \right\}.
\]
This has the same three-part structure as \eqref{eq:nadin}: the discrete
Dirichlet energy corresponds to
\(\int_C\nabla\alpha\cdot A\nabla\alpha\), the negative potential (or
linear-growth) term to \(-\int_C\mu\alpha^2\), and \(-\cE_p^c(a)\) to the
negative effective-diffusivity term
\(-\lambda^2|C|D_e(\alpha^2A)\).  At a fixed lattice scale the last term is
an exact nonlinear periodic correction built from \(\cosh\), rather than the
quadratic cell term appearing in the continuous formula.

In one dimension, take \(N=N_1\), \(c_{j,1}=d\), and write
\[
 \cP_N=\left\{a\in(0,\infty)^N:\sum_ja_j^2=1\right\},
 \qquad
 \cZ_N=\left\{\beta\in\R^N:\sum_j\beta_j=0\right\}.
\]
For a total phase \(\sigma\in\R\), define
\begin{equation}\label{eq:T-def}
 \cT_\sigma(a)
 :=\min_{\substack{\theta\in\R^N\\\sum_j\theta_j=\sigma}}
    \sum_{j=0}^{N-1}a_ja_{j+1}\cosh\theta_j.
\end{equation}

\begin{corollary}[One-dimensional formula]\label{cor:onedim}
For
\[
 (\cL_{p,V}u)_j
 =d(e^{-p}u_{j+1}+e^pu_{j-1}-2u_j)+V_ju_j,
\] which is defined in \eqref{oneDdiffu},
one has
\begin{equation}\label{eq:onedim}
 \Lambda_p(V)
 =\max_{a\in\cP_N}
 \left\{
  2d\,\cT_{Np}(a)-2d+\sum_jV_ja_j^2
 \right\}.
\end{equation}
\end{corollary}

To state the rearrangement result, let \(x_1\geq\cdots\geq x_N\) be the
entries of a vector, with multiplicity, and define
\begin{equation}\label{eq:bell}
 \rho_N=
 \begin{cases}
 (1,2,4,\ldots,2M,\,2M-1,2M-3,\ldots,3),&N=2M,\\
 (1,2,4,\ldots,2M,\,2M+1,2M-1,\ldots,3),&N=2M+1.
 \end{cases}
\end{equation}
The \emph{bell-shaped rearrangement}, also called the \emph{discrete
symmetric-decreasing rearrangement}, \(x^\#\) places
\(x_{\rho_N(1)},\ldots,x_{\rho_N(N)}\) consecutively around the cycle.
Rotations and reflections are equivalent.

\begin{theorem}[One-dimensional bell-shaped rearrangement]\label{thm:rearrangement}
For every \(a\in\cP_N\), \(\sigma\in\R\), \(p\in\R\), and \(V\in\R^N\),
\begin{equation*}
  \cT_\sigma(a^\#)\geq\cT_\sigma(a),
  \qquad
  \Lambda_p(V^\#)\geq\Lambda_p(V).
\end{equation*}
\end{theorem}

\medskip
\noindent\emph{Interpretation of the one-dimensional result.}
For the operator in Corollary~\ref{cor:onedim}, fix the multiset of values
of \(V\) and vary only their cyclic order.  Theorem~\ref{thm:rearrangement}
selects one bell-shaped order that is optimal for every tilt \(p\).  At
\(p=0\), this is consistent with the usual Rayleigh--Ritz,
Hardy--Littlewood, and discrete P\'olya--Szeg\H{o} principles.  For
\(p\ne0\), however, the periodic correction depends on the neighboring
products \(a_ja_{j+1}\), which are not preserved by rearrangement.  The
proof therefore passes to the scalar-flux formula: on a cycle the
divergence-free flux is constant, and the resulting four-point inequality
allows Supnick's theorem to select the bell order.  This mechanism is
specific to one dimension; no multidimensional rearrangement is asserted.

\medskip
\noindent\emph{Application to periodic lattice KPP propagation.}
Consider the one-dimensional periodic lattice KPP equation
\begin{equation}\label{eq:kpp-equation}
 \dot u_j(t)=d\bigl(u_{j+1}(t)+u_{j-1}(t)-2u_j(t)\bigr)+f_j(u_j(t)),
 \qquad j\in\Z,
\end{equation}
where \(d>0\), \(f_{j+N}=f_j\), and
\[
 f_j(0)=f_j(1)=0,\qquad 0<f_j(s)\leq f_j'(0)s\quad(0<s<1).
\]
Set \(V_j=f_j'(0)\).  The linearization of \eqref{eq:kpp-equation} at
\(u=0\) is
\begin{equation}\label{eq:kpp-linearized}
 \dot v_j=(\cL_{0,V}v)_j
 :=d(v_{j+1}+v_{j-1}-2v_j)+V_jv_j.
\end{equation}
For \(p\in\R\), the associated exponentially tilted periodic operator and
its principal eigenvalue are defined by
\begin{equation}\label{eq:kpp-tilted}
 \begin{gathered}
  (\cL_{p,V}\phi)_j
  :=d(e^{-p}\phi_{j+1}+e^p\phi_{j-1}-2\phi_j)+V_j\phi_j,\\
  \cL_{p,V}\phi=\Lambda_p(V)\phi,
  \qquad \phi_{j+N}=\phi_j,\quad \phi_j>0.
 \end{gathered}
\end{equation}
A right-moving pulsating front has the form
\(u_j(t)=U(j-ct,j)\), with \(U(\xi,j+N)=U(\xi,j)\), and connects \(1\) to
\(0\).  Under the standard periodic KPP hypotheses, the minimal front speed
and the spreading speed for compactly supported nontrivial data coincide
and are linearly determined by
\begin{equation}\label{eq:kpp-speed}
 c_{\rm tw}^*(V)=c_{\rm sp}^*(V)
 =:c^*(V)=\inf_{p>0}\frac{\Lambda_p(V)}{p}.
\end{equation}
Indeed, substituting the leading-edge ansatz
\(u_j(t)\sim e^{-p(j-ct)}\phi_j\) into \eqref{eq:kpp-linearized} gives
\(pc\,\phi=\cL_{p,V}\phi\).  Periodic lattice fronts and spectral
spreading formulas are developed in
\cite{GuoHamel2006,Cheng2015,LiangZhao2010, LiangZhou2020}; the broader continuous and
abstract background includes
\cite{Fisher1937,KPP1937,BerestyckiHamelNadirashvili2005,
Weinberger2002,LiangZhao2010}.

Nadin's continuous rearrangement theorem was itself motivated by this
propagation problem.  In his sign convention, Schwarz rearrangement lowers
\(k_\lambda(\mu)\) for every \(\lambda\) and therefore increases the KPP
minimal speed.  Our result is the exact periodic-lattice counterpart: the
bell order increases \(\Lambda_p(V)\) for every \(p\), so \eqref{eq:kpp-speed}
immediately gives the speed optimization.  This ordering problem should be
distinguished from varying the values of the growth coefficient.  In
\citet{LiangLinMatano2010} the admissible class consists of nonnegative
one-dimensional measures of prescribed mass, while
\citet{LiangMatano2014} treats concentration of the coefficient in
two-dimensional stratified media.  Those works share the KPP motivation but
use different admissible classes and optimization mechanisms.

\begin{corollary}[Optimal one-dimensional KPP speeds]\label{cor:speeds}
For the periodic KPP equation \eqref{eq:kpp-equation}, whose linearization
and tilted principal eigenvalue problem are \eqref{eq:kpp-linearized} and
\eqref{eq:kpp-tilted}, the speeds characterized by \eqref{eq:kpp-speed}
satisfy
\begin{equation}\label{eq:speed-rearrangement}
 c^*(V^\#)\geq c^*(V).
\end{equation}
Thus, among all cyclic orders of a fixed multiset of linear growth rates,
the bell order maximizes both the minimal front speed and the spreading
speed, whose common value is denoted by \(c^*\) in \eqref{eq:kpp-speed}.
\end{corollary}

\subsection{Supplementary physical background: the Hatano--Nelson model}

Up to the Hamiltonian sign convention, the one-dimensional tilted diffusion
operator is the real periodic Hatano--Nelson model
\cite{HatanoNelson1996,HatanoNelson1997}; its ground-state energy is
\(-\Lambda_p(V)\).  The positive right and left eigenvectors \(\phi,\psi\)
define the biorthogonal density \(\phi_x\psi_x\), while half their logarithmic
ratio is precisely the gauge variable in our cell problem
\cite{HatanoNelson1998}.  This interpretation is useful background, whereas
the statements and proofs below use the language of principal eigenvalues,
periodic correctors, and discrete diffusion.

 \subsection {Organization of the paper} Theorem~\ref{thm:highdim} and Corollaries~\ref{cor:reciprocal} and
\ref{cor:onedim} are proved in Section~\ref{sec:variational}.  The proof of
the main theorem is a direct diagonal symmetrization argument: the positive
left and right eigenvectors identify a saddle point of the profile--corrector
functional.  The one-dimensional reduction is then combined with Supnick's
symmetric Monge tour in
Section~\ref{sec:rearrangement}.  Supnick's theorem is an existing
combinatorial result; our use of it avoids reproving a more cumbersome
cyclic exchange argument.  The open-chain Dirichlet and Neumann variants
are discussed in Proposition~\ref{prop:open-chain}.  For completeness,
Appendix~\ref{app:flux-duality} records the more general divergence-free
flux representation and its convex-duality proof; it is not needed for the
direct proof of the main theorem.

\section{The high-dimensional variational principle}
\label{sec:variational}

\subsection{The periodic corrector problem}

For \(a\in\cP_{\mathbf N}\), denote
\[
 w_{x,\nu}:=\kappa_{x,\nu}a_xa_{x+e_\nu},\qquad
 \theta_{x,\nu}(\beta)
 :=\beta_{x+e_\nu}-\beta_x+\gamma_{x,\nu}.
\]
The corrector-dependent part of \(\cF_{\cM}\) is
\[
 \cC_a(\beta)
 :=2\sum_{x,\nu}w_{x,\nu}\cosh\theta_{x,\nu}(\beta).
\]

\begin{lemma}[Nonlinear periodic corrector problem]\label{lem:cell}
For every \(a\in\cP_{\mathbf N}\), the functional
\(\cC_a\) has a unique minimizer in \(\cZ_{\mathbf N}\).  A corrector
\(\widehat\beta\) is the minimizer if and only if
\begin{equation}\label{eq:current}
 \divd J(x)
 :=\sum_{\nu=1}^m
    \bigl(J_{x,\nu}-J_{x-e_\nu,\nu}\bigr)=0,
 \qquad
 J_{x,\nu}
 =w_{x,\nu}\sinh\theta_{x,\nu}(\widehat\beta).
\end{equation}
\end{lemma}

\begin{proof}
Let \(\beta,h\in\cZ_{\mathbf N}\).  The Hessian in the direction \(h\) is
\[
 2\sum_{x,\nu}w_{x,\nu}
 \cosh\theta_{x,\nu}(\beta)
 (h_{x+e_\nu}-h_x)^2.
\]
It is positive unless \(h\) is constant because the torus graph is
connected.  Since the only constant vector in \(\cZ_{\mathbf N}\) is zero,
the functional is strictly convex on \(\cZ_{\mathbf N}\). 
 It is also coercive on \(\mathcal Z_{\mathbf N}\). Indeed, let\[\operatorname{osc}(\beta):=\max_x\beta_x-\min_x\beta_x.\]Since the torus graph is finite and connected, vertices at which the maximum and minimum are attained can be joined by a nearest-neighbor path of length at most \(\operatorname{diam}(\mathbb T_{\mathbf N})\). Telescoping along this path gives \[\max_{x,\nu}\left|\beta_{x+e_\nu}-\beta_x\right|\geq\frac{\operatorname{osc}(\beta)}{\operatorname{diam}(\mathbb T_{\mathbf N})}.\]Moreover, on the mean-zero space \(\mathcal Z_{\mathbf N}\), \(\|\beta\|\to\infty\) implies \(\operatorname{osc}(\beta)\to\infty\). Since the edge weights \(w_{x,\nu}\) are positive and the shifts \(\gamma_{x,\nu}\) are fixed, at least one term in \(\mathcal C_a(\beta)\) therefore tends to \(+\infty\). Hence\[\|\beta\|\to\infty\quad\Longrightarrow\quad\mathcal C_a(\beta)\to\infty.\]  and then a unique
minimizer exists.  Differentiation with respect to \(\beta_x\) gives
twice the incoming flux minus twice the outgoing flux, which is
equivalent to \eqref{eq:current}.
\end{proof}

\subsection{ Symmetrization and proof}

\begin{proof}[Proof of Theorem~\ref{thm:highdim}]
Let \(\phi,\psi\gg0\) be positive right and left eigenvectors, normalized
by \(\sum_x\phi_x\psi_x=1\).  A reciprocal rescaling of \(\phi\) and
\(\psi\) preserves this normalization, so we also impose
\(\sum_x\log(\phi_x/\psi_x)=0\).  Set
\[
 a_x:=\sqrt{\phi_x\psi_x},\qquad
 \beta_x:=\frac12\log\frac{\phi_x}{\psi_x},\qquad
 S:=\diag(e^{\beta_x}).
\]
Then \(\phi=Sa\), \(\psi=S^{-1}a\), and for
\(B:=S^{-1}\cM S\),
\[
  Ba=\Lambda(\cM)a,\qquad
  B^{\transpose}a=\Lambda(\cM)a.
\]
Consequently,
\[
 Q:=\frac12(B+B^{\transpose})
\]
is a symmetric irreducible Metzler matrix satisfying
\(Qa=\Lambda(\cM)a\).  Since \(a>0\), this is the largest eigenvalue of
\(Q\).

For an arbitrary positive unit vector \(b\), pairing the two directed
terms on the edge \(\{x,x+e_\nu\}\) gives
\[
\begin{aligned}
 b^{\transpose}Bb
 =\sum_xq_xb_x^2
 +2\sum_{x,\nu}\kappa_{x,\nu}b_xb_{x+e_\nu}
 \cosh\!\bigl(
  \beta_{x+e_\nu}-\beta_x+\gamma_{x,\nu}
 \bigr)
 =\cF_{\cM}(b,\beta).
\end{aligned}
\]
Because a real quadratic form only sees the symmetric part,
\[
 \cF_{\cM}(b,\beta)=b^{\transpose}Qb
 \leq a^{\transpose}Qa=\Lambda(\cM).
\]

It remains to identify \(\beta\) as the minimizing corrector for the
profile \(a\).  If \(h\in\cZ_{\mathbf N}\) and
\(R_h:=\diag(h_x)\), differentiation at \(t=0\) gives
\[
 \frac{d}{dt}\cF_{\cM}(a,\beta+th)\bigg|_{t=0}
 =a^{\transpose}(BR_h-R_hB)a.
\]
Since \(Ba=B^{\transpose}a=\Lambda(\cM)a\), the two terms on the
right are both
\(\Lambda(\cM)\sum_xa_x^2h_x\), and the derivative vanishes.
Lemma~\ref{lem:cell} now shows, using only ordinary strict convexity, that
\(\beta\) is the unique minimizer in \(\cZ_{\mathbf N}\).

We have therefore constructed a genuine saddle:
\[
 \cF_{\cM}(b,\beta)
 \leq\cF_{\cM}(a,\beta)
 \leq\cF_{\cM}(a,\eta)
\]
for every \(b\in\cP_{\mathbf N}\) and every \(\eta\in\cZ_{\mathbf N}\).
The left inequality shows that
\(\max_b\cF_{\cM}(b,\beta)=\Lambda(\cM)\), whereas the right one shows
that \(\min_\eta\cF_{\cM}(a,\eta)=\Lambda(\cM)\).  For an arbitrary
\(\eta\), the maximum over \(b\) is at least
\(\cF_{\cM}(a,\eta)\geq\Lambda(\cM)\); for an arbitrary \(b\), the
minimum over \(\eta\) is at most
\(\cF_{\cM}(b,\beta)\leq\Lambda(\cM)\).  These observations prove both
extremal formulas in \eqref{eq:highdim} and their equality, without an
abstract minimax theorem.  Since
\(\max_{b\in\cP_{\mathbf N}}\cF_{\cM}(b,\eta)
=\lambda_{\max}(H_\eta)\), the first formula follows as well.  The flux
assertion is Lemma~\ref{lem:cell}.
\end{proof}

\begin{proof}[Proof of Corollary~\ref{cor:reciprocal}]
For symmetric edge diffusion,
\[
 \kappa_{x,\nu}=c_{x,\nu},\qquad
 \gamma_{x,\nu}(p)=-p_\nu.
\]
At zero phase,
\[
\begin{aligned}
 &\sum_x\left[
 V(x)-\sum_\nu(c_{x,\nu}+c_{x-e_\nu,\nu})
 \right]a_x^2
 +2\sum_{x,\nu}c_{x,\nu}a_xa_{x+e_\nu}\\
 &\hspace{20mm}
 =-\sum_{x,\nu}c_{x,\nu}(a_{x+e_\nu}-a_x)^2
  +\sum_xV(x)a_x^2.
\end{aligned}
\]
Separating \(\cosh\theta=1+(\cosh\theta-1)\) in
\eqref{eq:highdim} proves \eqref{eq:reciprocal}.
\end{proof}

\begin{remark}[Relation with Donsker--Varadhan--Friedland \cite{DonskerVaradhan1975,DonskerVaradhan1976,Friedland1981}]
After adding a scalar multiple of the identity if necessary, the
classical finite-state formula is
\[
 \Lambda(\cM)
 =\sup_{\pi\in\Delta}\inf_{u>0}
   \sum_x\pi_x\frac{(\cM u)_x}{u_x}.
\] where
\(\Delta:=\{\pi\in[0,\infty)^{\mathbb T_{\mathbf N}}:
\sum_x\pi_x=1\}\)
is the probability simplex on the periodic cell. 
The substitutions \(\pi_x=a_x^2\) and \(u_x=a_xe^{\beta_x}\) give the
maximum--minimum member of \eqref{eq:highdim} after pairing reverse edges;
multiplying \(u\) by a constant allows one to take
\(\beta\in\cZ_{\mathbf N}\).  The direct proof above additionally
identifies the left--right eigenvector saddle, proves the reverse extremal
order, and exhibits the symmetric conjugated operator and conserved
nonlinear flux.
\end{remark}

\subsection{The one-dimensional reduction}

On the cycle, every divergence-free flux is constant.  This makes
\eqref{eq:T-def} explicit.

\begin{lemma}[Constant-flux reduction]\label{lem:constant-current}
For \(a\in\cP_N\) and \(\sigma\in\R\), there is a unique scalar flux
\(J(a,\sigma)\) such that
\[
 \sum_{j=0}^{N-1}
 \arsinh\left(\frac{J(a,\sigma)}{a_ja_{j+1}}\right)=\sigma.
\]
The unique minimizing phases in \eqref{eq:T-def} are
\[
 \theta_j=\arsinh\left(\frac{J(a,\sigma)}{a_ja_{j+1}}\right),
\]
and
\[
 \cT_\sigma(a)
 =\sum_j\sqrt{a_j^2a_{j+1}^2+J(a,\sigma)^2},
 \qquad
 \cT_{-\sigma}(a)=\cT_\sigma(a).
\]
\end{lemma}

\begin{proof}
The objective in \eqref{eq:T-def} is strictly convex and coercive on
the affine hyperplane \(\sum_j\theta_j=\sigma\).  Its Lagrange equations are
\[
 a_ja_{j+1}\sinh\theta_j=J.
\]
The displayed scalar equation is continuous and strictly increasing in
\(J\), from \(-\infty\) to \(+\infty\), and hence determines \(J\)
uniquely.  The value and evenness follow immediately.
\end{proof}

\medskip
\noindent\emph{Continuous and discrete conserved quantities.}
For the continuous periodic operator
\[
 -L_\lambda\phi=-\phi''-2\lambda\phi'-(\mu+\lambda^2)\phi,
\]
let \(\phi\) and \(\psi\) be the positive eigenfunctions of the operator
and its adjoint, and set
\(\alpha=\sqrt{\phi\psi}\) and
\(\beta=\tfrac12\log(\phi/\psi)\).  Subtraction of the two transformed
eigenvalue equations gives
\(\alpha^2(\beta'+\lambda)=J\).  If the period is \(\ell\), then
\[
 J=\frac{\lambda\ell}{\displaystyle\int_0^\ell\alpha^{-2}},
 \qquad
 \int_0^\ell\alpha^2(\beta'+\lambda)^2
 =\frac{\lambda^2\ell^2}
        {\displaystyle\int_0^\ell\alpha^{-2}}.
\]
The reciprocal integral is the one-dimensional effective diffusivity in
\eqref{eq:nadin}.  Lemma~\ref{lem:constant-current} is its exact lattice
counterpart: the linear reciprocal relation is replaced by a sum of
\(\arsinh(J/(a_ja_{j+1}))\), and the correction term by a sum of
\(\sqrt{a_j^2a_{j+1}^2+J^2}\).  Only in the small-phase continuum limit do
these nonlinear expressions reduce to the continuous linear flux relation
and quadratic energy.
\medskip

For \(w>0\) and a scalar flux \(J\in\R\), set
\[
 g_J(w):=\sqrt{w^2+J^2}
 -J\,\arsinh\left(\frac Jw\right).
\]

\begin{lemma}[Scalar-flux formula]\label{lem:scalar-flux}
For every \(a\in\cP_N\) and \(\sigma\in\R\),
\begin{equation}\label{eq:scalar-flux}
 \cT_\sigma(a)
 =\sup_{J\in\R}
 \left\{J\sigma+\sum_jg_J(a_ja_{j+1})\right\}.
\end{equation}
\end{lemma}

\begin{proof}
For every \(w>0\) and \(J,\theta\in\R\), elementary one-variable
calculus gives
\[
 w\cosh\theta\geq J\theta+g_J(w),
\]
with equality exactly when \(w\sinh\theta=J\).  Summing this inequality
under \(\sum_j\theta_j=\sigma\) gives
\[
 \cT_\sigma(a)\geq
 J\sigma+\sum_jg_J(a_ja_{j+1})
\]
for every \(J\).  Choose the constant flux \(J=J(a,\sigma)\) from
Lemma~\ref{lem:constant-current}.  Its minimizing phases satisfy the
equality condition on every edge, so the displayed lower bound is exact.
This proves \eqref{eq:scalar-flux} directly.
\end{proof}

\begin{proof}[Proof of Corollary~\ref{cor:onedim}]
For \(\beta\in\cZ_N\), put
\(\theta_j=\beta_{j+1}-\beta_j-p\).  Then
\(\sum_j\theta_j=-Np\).  Conversely, every phase vector with this sum
comes from a unique \(\beta\in\cZ_N\).  Theorem
\ref{thm:highdim} and the evenness of \(\cT_\sigma\) give
\eqref{eq:onedim}.
\end{proof}

\section{One-dimensional rearrangement and propagation optimization}
\label{sec:rearrangement}

\subsection{Supnick's theorem and a cyclic kernel inequality}

A symmetric matrix \(C=(c_{rs})_{1\leq r,s\leq N}\) is a
\emph{symmetric Monge matrix} if
\[
 c_{rs}+c_{tu}\leq c_{ru}+c_{ts}
 \qquad(r<t,\ s<u).
\]
Supnick proved that the symmetric Monge traveling-salesman problem has a
fixed optimal tour \cite{Supnick1957}; see also
\cite{BurkardEtAl1998}.  In the notation of \eqref{eq:bell}, the result
is as follows.

The word ``fixed'' is essential.  Once the cities have been labeled in
Monge order, the same Hamiltonian cycle is optimal for every symmetric
Monge cost matrix: in the conventional orientation it visits the odd
indices increasingly and then the even indices decreasingly.  Up to
reversal and rotation, this is exactly the bell tour \(\rho_N\).  Thus
the theorem is stronger than a polynomial-time algorithm; it identifies
the optimizer without using the numerical entries beyond their Monge
ordering.  From a modern viewpoint it is an early and canonical
structural solvability theorem for the traveling-salesman problem.

\begin{theorem}[Supnick's fixed-tour theorem \cite{Supnick1957}]\label{thm:supnick}
If \(C\) is symmetric Monge, the bell tour \(\rho_N\) minimizes
\[
 \sum_{j=0}^{N-1}c_{\pi_j,\pi_{j+1}},\qquad \pi_N=\pi_0,
\]
over all Hamiltonian cycles \(\pi\).
\end{theorem}

Supnick's result is now a standard reference point in the theory of
well-solvable traveling-salesman classes.  Later work placed it among
pyramidal, Demidenko, Kalmanson, and related Monge-type problems, with
applications to sequencing, scheduling, warehouse routing, and
geometric routing; see the survey \cite{BurkardEtAl1998} and the
multi-stripe extension \cite{CelaDeinekoWoeginger2017}.  It does not
solve an arbitrary traveling-salesman problem.  Its force comes from
the rigid Monge four-point inequalities.  In the present paper its role
is analytic rather than algorithmic: it converts a four-point inequality
for an edge kernel into a sharp cyclic rearrangement inequality.

A symmetric kernel \(K:I^2\to\R\) will be assumed to satisfy the
four-point inequality
\[
 K(x_1,y_1)+K(x_2,y_2)
 \geq K(x_1,y_2)+K(x_2,y_1)
\]
whenever \(x_1\geq x_2\) and \(y_1\geq y_2\).  For \(C^2\) kernels,
this follows from \(\partial_{xy}^2K\geq0\).

\begin{corollary}[Cyclic rearrangement under a four-point condition]\label{cor:cyclic}
Let \(x_1\geq\cdots\geq x_N\) lie in \(I\), and let \(K\) be symmetric
and satisfy the preceding four-point inequality.  Then every cyclic order \(\pi\) satisfies
\begin{equation*}
 \sum_{j=1}^{N}K(x_{\rho_N(j)},x_{\rho_N(j+1)})
 \geq\sum_{j=1}^{N}K(x_{\pi_j},x_{\pi_{j+1}}),
\end{equation*}
where \(\rho_N(N+1)=\rho_N(1)\) and \(\pi_{N+1}=\pi_1\).
\end{corollary}

\begin{proof}
Set \(c_{rs}=-K(x_r,x_s)\).  The four-point inequality says exactly that \(C\) is
a symmetric Monge matrix.  Theorem~\ref{thm:supnick} minimizes the
cyclic sum of \(c_{rs}\), and therefore maximizes the cyclic sum of
\(K\).
\end{proof}

\begin{remark}
The corollary is an application of Supnick's theorem, not a new
traveling-salesman result.  It replaces a longer sequence of cyclic
exchange arguments.  General Schwarz and P\'olya--Szeg\H{o}
inequalities on lattices and graphs concern a broader geometry
\cite{Steinerberger2024,HajaiejHanHua2026}.  Pendulum arrangements and
longer cyclic interactions lead to related but different optimization
problems
\cite{CasselMannorTennenholtz2020,HolmesHolroydRamirez2023,
CelaDeinekoWoeginger2017}.
\end{remark}

\subsection{Rearrangement of the periodic correction}

\begin{proof}[Proof of the first inequality in
Theorem~\ref{thm:rearrangement}]
For a fixed scalar flux \(J\in\R\), define
\[
 K_J(x,y):=g_J(xy),\qquad x,y>0.
\]
A direct calculation gives
\[
 \frac{\partial^2K_J}{\partial x\,\partial y}(x,y)
 =\frac{xy}{\sqrt{x^2y^2+J^2}}>0.
\]
Thus \(K_J\) is symmetric and satisfies the strict four-point inequality.  Corollary
\ref{cor:cyclic} yields
\[
 \sum_jg_J(a_j^\#a_{j+1}^\#)
 \geq\sum_jg_J(a_ja_{j+1})
\]
for every \(J\).  Taking the supremum in the scalar-flux formula
\eqref{eq:scalar-flux} proves
\(\cT_\sigma(a^\#)\geq\cT_\sigma(a)\).
\end{proof}

At \(\sigma=0\), \(g_0(w)=w\), so the preceding proof reduces to
\[
 \sum_ja_j^\#a_{j+1}^\#\geq\sum_ja_ja_{j+1},
\]
equivalently the cyclic discrete P\'olya--Szeg\H{o} inequality.  The
nonzero-tilt result is stronger because it controls the whole family of
kernels \(K_J\) with positive mixed derivative.

\subsection{Principal eigenvalues}
First, we have
\begin{lemma}[Finite Hardy--Littlewood inequality \cite{HardyLittlewoodPolya1952} ]\label{lem:HL}
For \(V\in\R^N\) and \(a\in\cP_N\),
\begin{equation}\label{eq:HL}
 \sum_jV_j^\#(a_j^\#)^2\geq\sum_jV_ja_j^2.
\end{equation}
\end{lemma}

\begin{proof}[Proof of the second inequality in
Theorem~\ref{thm:rearrangement}]
For every \(a\in\cP_N\), the first part of the theorem and Lemma
\ref{lem:HL} give
\[
 2d\,\cT_{Np}(a^\#)-2d+\sum_jV_j^\#(a_j^\#)^2
 \geq
 2d\,\cT_{Np}(a)-2d+\sum_jV_ja_j^2.
\]
The left-hand side is bounded above by \(\Lambda_p(V^\#)\) through
\eqref{eq:onedim}.  Maximizing the right-hand side over \(a\) proves
\(\Lambda_p(V^\#)\geq\Lambda_p(V)\).
\end{proof}

The transpose relation
\[
 \cL_{p,V}^{\transpose}=\cL_{-p,V}
\]
implies \(\Lambda_p(V)=\Lambda_{-p}(V)\), even when \(V\) is not
reflection invariant.

\subsection{Minimal front speed and spreading speed}

In one dimension, a right-moving front is
\[
 u_j(t)=U(j-ct,j),\qquad U(\xi,j+N)=U(\xi,j),
\]
with \(U(-\infty,j)=1\) and \(U(+\infty,j)=0\).  The right spreading
speed for compactly supported nonzero data is the threshold
\(c^*\) characterized by decay for \(j\geq ct\) when
\(c>c^*\), and convergence to \(1\) on
\(0\leq j\leq ct\) when \(0<c<c^*\).  It equals the minimal front speed,
and periodic lattice KPP
theory gives the characterization \eqref{eq:kpp-speed}, with
\(\Lambda_p(V)\) defined by \eqref{eq:kpp-tilted}.

\begin{proof}[Proof of Corollary~\ref{cor:speeds}]
Theorem~\ref{thm:rearrangement} gives
\[
 \frac{\Lambda_p(V^\#)}p\geq\frac{\Lambda_p(V)}p
 \qquad(p>0).
\]
Taking the infimum over \(p\) and using \eqref{eq:kpp-speed} proves
\eqref{eq:speed-rearrangement}.  Every cyclic arrangement of the entries
of \(V\) has the same bell rearrangement, so this is optimal among all
permutations.
\end{proof}

\begin{remark}[Scope of the speed conclusion]
This is a KPP, or pulled-front, statement: the speed is determined by
the linear growth rates \(V_j=f_j'(0)\).  It need not extend to pushed
or bistable fronts.  It also differs from optimizing the numerical
values of \(V_j\) under a mass constraint, as in
\cite{LiangLinMatano2010,LiangMatano2014}.  No multidimensional
rearrangement claim is made; only the variational principle is
high-dimensional.
\end{remark}

\subsection{Open chains: Dirichlet and Neumann boundary conditions}

The role of the exponential tilt is different on an open chain.  Let
\(I_N=\{1,\ldots,N\}\), fix \(d>0\), and first impose Dirichlet boundary
conditions by
\[
 (\cL^D_{p,V}u)_j
 =d(e^{-p}u_{j+1}+e^pu_{j-1}-2u_j)+V_ju_j,
 \qquad u_0=u_{N+1}=0.
\]
For the reflecting, or discrete Neumann, realization, set
\[
 (\cL^N_{p,V}u)_j=
 \begin{cases}
 d(e^{-p}u_2-u_1)+V_1u_1,&j=1,\\
 d(e^{-p}u_{j+1}+e^pu_{j-1}-2u_j)+V_ju_j,
       &2\leq j\leq N-1,\\
 d(e^pu_{N-1}-u_N)+V_Nu_N,&j=N.
 \end{cases}
\]
Let \(S_p=\diag(e^{p},e^{2p},\ldots,e^{Np})\).  Since an open path has
no nontrivial cycle,
\[
 S_p^{-1}\cL^D_{p,V}S_p=\cL^D_{0,V},
 \qquad
 S_p^{-1}\cL^N_{p,V}S_p=\cL^N_{0,V}.
\]
Consequently, if \(\Lambda_p^D(V)\) and \(\Lambda_p^N(V)\) denote the
corresponding principal eigenvalues, then they are independent of the
tilt and have the Rayleigh representations
\[
\begin{aligned}
 \Lambda_p^D(V)
 &=\max_{a\in\cP_N}
 \left\{-d\sum_{j=0}^{N}(a_{j+1}-a_j)^2
             +\sum_{j=1}^{N}V_ja_j^2\right\},
 &a_0=a_{N+1}=0,\\
 \Lambda_p^N(V)
 &=\max_{a\in\cP_N}
 \left\{-d\sum_{j=1}^{N-1}(a_{j+1}-a_j)^2
             +\sum_{j=1}^{N}V_ja_j^2\right\}.
\end{aligned}
\]
Thus there is no nonlinear flux-correction problem in these open-boundary
realizations.  The tilt is an exact discrete gradient and can be removed by
diagonal conjugation; on the periodic cycle its total phase around the cycle
is generally nonzero, so no periodic diagonal conjugation removes it.

The two boundary conditions lead to different optimal arrangements.
For the Dirichlet problem, rotate the bell tour \(\rho_{N+1}\) from
\eqref{eq:bell} so that the rank \(N+1\) comes first, delete that rank,
and denote the remaining order by \(\rho_N^D\).  If
\(x_1\geq\cdots\geq x_N\), the Dirichlet bell rearrangement \(x^{D\#}\)
places
\(x_{\rho_N^D(1)},\ldots,x_{\rho_N^D(N)}\) along the path.  It clusters
the largest entries near the center and the smallest entries near the
two endpoints.  For the Neumann problem, let \(x^\downarrow\) denote
the monotone decreasing order along the path; its reflection is
equivalent.

\begin{proposition}[Open-chain rearrangements]\label{prop:open-chain}
For every \(V\in\R^N\) and \(p\in\R\),
\[
 \Lambda_p^D(V^{D\#})\geq\Lambda_p^D(V),
 \qquad
 \Lambda_p^N(V^\downarrow)\geq\Lambda_p^N(V).
\]
\end{proposition}

\begin{proof}
For the Dirichlet energy, append the value zero to the entries of
\(a\).  The sum
\[
 \sum_{j=0}^{N}(a_{j+1}-a_j)^2,\qquad a_0=a_{N+1}=0,
\]
is the cyclic squared-distance cost through these \(N+1\) values.
Equivalently, since \(\sum_ja_j^2\) is fixed, minimizing this cost is
the same as maximizing the cyclic sum of neighboring products.
Supnick's fixed-tour theorem therefore gives
\[
 \sum_{j=0}^{N}(a^{D\#}_{j+1}-a^{D\#}_j)^2
 \leq\sum_{j=0}^{N}(a_{j+1}-a_j)^2.
\]
The finite Hardy--Littlewood inequality pairs \(V^{D\#}\) and
\((a^{D\#})^2\) in the same rank order, so their potential term is no
smaller.  The Dirichlet conclusion follows from the Rayleigh formula.

For the Neumann energy, ordering the entries of \(a\) monotonically
removes every crossing and minimizes
\(\sum_{j=1}^{N-1}(a_{j+1}-a_j)^2\).  Applying Hardy--Littlewood to
\(V^\downarrow\) and \((a^\downarrow)^2\), and then maximizing the
Neumann Rayleigh formula, proves the second inequality
\cite{HardyLittlewoodPolya1952,Supnick1957}.
\end{proof}

\begin{remark}[Reflection heuristic]
At an intuitive level, a reflecting Neumann chain may be unfolded across
its two endpoints into an even periodic chain.  In this picture, the
monotone optimal arrangement on the path is one half of a symmetric
clustered arrangement on the doubled cycle.  The exact treatment of the
endpoint sites and edge weights depends on the chosen discrete Neumann
convention, so we use this observation only as a geometric interpretation;
the proof above works directly with the path Rayleigh formula.
\end{remark}

\begin{remark}[Boundary convention and the continuous analogue]
The Neumann statement above uses the reflecting matrix realization,
which is compatible with the diagonal conjugation \(S_p\).  This convention
must be distinguished from inserting a first-order exponential twist
into a continuous differential expression while retaining the
unmodified condition \(\phi'=0\) at both endpoints.  Under exponential
conjugation, an ordinary Neumann condition becomes a
\(p\)-dependent Robin condition.  Keeping the unmodified Neumann
condition therefore defines a different, genuinely \(p\)-dependent
problem, and the preceding similarity and rearrangement conclusions do
not automatically apply.  Variational principles for asymmetric
operators with natural boundary conditions go back to
\citet{Holland1978}; continuous Dirichlet rearrangement problems for
general elliptic operators are treated, for example, by
\cite{HamelNadirashviliRuss2011}.  Discrete rearrangements on more
general graphs require additional geometric hypotheses
\cite{Steinerberger2024,HajaiejHanHua2026}.
\end{remark}

\appendix

\section{A divergence-free flux representation by convex duality}
\label{app:flux-duality}

The direct saddle-point proof of Theorem~\ref{thm:highdim} does not require
convex duality.  We nevertheless record here a complementary representation
of the inner corrector problem.  Orient every edge from \(x\) to
\(x+e_\nu\), and use the notation \(w_{x,\nu}\), \(\cC_a\), and
\(\divd J\) from Section~\ref{sec:variational}.

\begin{proposition}[Divergence-free flux representation]
\label{prop:dual-current}
For fixed \(a\in\cP_{\mathbf N}\),
\begin{equation*}
\begin{aligned}
 \min_{\beta\in\cZ_{\mathbf N}}\cC_a(\beta)
 =2\sup_{\divd J=0}\sum_{x,\nu}\Bigg[
 &J_{x,\nu}\gamma_{x,\nu}
 +\sqrt{J_{x,\nu}^2+w_{x,\nu}^2}\\
 &-J_{x,\nu}\arsinh
   \left(\frac{J_{x,\nu}}{w_{x,\nu}}\right)
 \Bigg].
\end{aligned}
\end{equation*}
At the optimum,
\[
 J_{x,\nu}
 =w_{x,\nu}\sinh\theta_{x,\nu}(\widehat\beta).
\]
\end{proposition}

\begin{proof}
For the oriented edge \(e=(x,x+e_\nu)\), write
\[
 D\beta_e=\beta_{x+e_\nu}-\beta_x,\qquad
 w_e=w_{x,\nu},\qquad \gamma_e=\gamma_{x,\nu}.
\]
The convex conjugate of \(f_e(z)=w_e\cosh z\) is
\[
 f_e^*(J_e)
 =\sup_{z\in\R}\{J_ez-f_e(z)\}
 =J_e\arsinh\left(\frac{J_e}{w_e}\right)
  -\sqrt{J_e^2+w_e^2},
\]
because the supremum is attained at
\(z=\arsinh(J_e/w_e)\).  The Fenchel--Moreau identity therefore gives,
edge by edge,
\[
 f_e(D\beta_e+\gamma_e)
 =\sup_{J_e\in\R}
 \{J_e(D\beta_e+\gamma_e)-f_e^*(J_e)\}.
\]
After summing over the edges, finite-dimensional
Fenchel--Rockafellar duality \cite[Chap.~III]{EkelandTemam1999} permits
minimization in \(\beta\).  The discrete summation-by-parts identity
\[
 \sum_{x,\nu}J_{x,\nu}
 (\beta_{x+e_\nu}-\beta_x)
 =-\sum_x\beta_x\divd J(x)
\]
shows that this minimization imposes \(\divd J=0\).  Indeed,
\(\sum_x\divd J(x)=0\) on the torus, so orthogonality to
\(\cZ_{\mathbf N}\) forces the discrete divergence to vanish.  Since every
\(f_e\) is finite and continuous on \(\R\), the standard
finite-dimensional constraint qualification holds and there is no duality
gap.  Consequently,
\[
 \min_{\beta\in\cZ_{\mathbf N}}
 \sum_e f_e(D\beta_e+\gamma_e)
 =\sup_{\divd J=0}\sum_e
 \{J_e\gamma_e-f_e^*(J_e)\}.
\]
Substituting the displayed expression for \(f_e^*\) and multiplying by
\(2\) proves the formula.  Finally, equality in the edgewise Fenchel
inequality is equivalent to
\[
 J_e=f_e'(D\widehat\beta_e+\gamma_e),
\]
which is precisely the asserted flux relation.
\end{proof}

We now make explicit how Proposition~\ref{prop:dual-current} enters the
principal-eigenvalue formula.  For \(a\in\cP_{\mathbf N}\) and a
divergence-free edge field \(J\), define
\[
\begin{aligned}
 \mathcal G_{\cM}(a,J)
 :=\sum_xq_xa_x^2+2\sum_{x,\nu}\Bigg[
 &J_{x,\nu}\gamma_{x,\nu}
 +\sqrt{J_{x,\nu}^2+
        \kappa_{x,\nu}^2a_x^2a_{x+e_\nu}^2}\\
 &-J_{x,\nu}\arsinh\left(
   \frac{J_{x,\nu}}
        {\kappa_{x,\nu}a_xa_{x+e_\nu}}\right)
 \Bigg].
\end{aligned}
\]

\begin{corollary}[Principal-eigenvalue flux formula]
\label{cor:principal-flux}
The principal eigenvalue satisfies
\begin{equation}\label{eq:principal-flux}
 \Lambda(\cM)
 =\max_{a\in\cP_{\mathbf N}}
   \sup_{\divd J=0}\mathcal G_{\cM}(a,J).
\end{equation}
Let \(\phi,\psi\) be normalized as in Theorem~\ref{thm:highdim}.  An
optimizing pair in \eqref{eq:principal-flux} is
\[
 a_x^*=\sqrt{\phi_x\psi_x}
\]
together with
\[
\begin{aligned}
 J_{x,\nu}^*
 &=
 \kappa_{x,\nu}a_x^*a_{x+e_\nu}^*
 \sinh\!\bigl(
   \beta_{x+e_\nu}^*-\beta_x^*+\gamma_{x,\nu}\bigr)\\
 &=\frac12\Bigl[
 r_\nu^+(x)\psi_x\phi_{x+e_\nu}
 -r_\nu^-(x+e_\nu)\phi_x\psi_{x+e_\nu}
 \Bigr],
\end{aligned}
\]
where
\(\beta_x^*=\tfrac12\log(\phi_x/\psi_x)\).
\end{corollary}

\begin{proof}
The third line of \eqref{eq:highdim} reads
\[
 \Lambda(\cM)
 =\max_{a\in\cP_{\mathbf N}}
 \left\{\sum_xq_xa_x^2+
       \min_{\beta\in\cZ_{\mathbf N}}\cC_a(\beta)\right\}.
\]
Substituting Proposition~\ref{prop:dual-current} gives
\eqref{eq:principal-flux}.  Theorem~\ref{thm:highdim} identifies the
maximizing profile \(a^*\) and its minimizing corrector \(\beta^*\);
the equality condition in Proposition~\ref{prop:dual-current} then gives
the first expression for \(J^*\).  Expanding the hyperbolic sine and using
\[
 a_x^*e^{\beta_x^*}=\phi_x,\qquad
 a_x^*e^{-\beta_x^*}=\psi_x
\]
gives the second expression.  The left and right eigenvalue equations
equivalently imply \(\divd J^*=0\).
\end{proof}

Thus the convex-duality argument is the dualization of the inner corrector
minimum in the maximum--minimum member of \eqref{eq:highdim}.  By itself it
does not justify reversing the two extremal operations.  That reversal is
proved in Section~\ref{sec:variational} by the explicit Perron saddle point.
At that saddle, however, the corrector \(\beta^*\) and the flux \(J^*\)
saturate every edgewise Fenchel inequality, so the direct and dual
descriptions give the same principal eigenvalue and the same optimizer.

    \section*{Declaration on the Use of AI}
The human author initiated and led the research program, formulated the mathematical problem, developed the main results, and independently verified all arguments. Regarding the variational formula, the human author independently completed the proof presented in Section 2 and subsequently discussed it with ChatGPT. ChatGPT suggested that the formula might also be proved using convex duality. Building on this suggestion, we developed a convex-duality-based proof, which is included in the appendix.

For the rearrangement part, the first draft established the result by direct calculation. ChatGPT provided a list of references related to rearrangement inequalities, summarized and compared the relevant results, and identified several works connected with the present problem, including Supnick’s fixed-tour theorem, of which we had previously been unaware. After further investigation, we found that Supnick’s theorem substantially simplified our original argument, and the final manuscript adopts this simplified proof.

ChatGPT also assisted with mathematical exposition, notation, grammar, and presentation. The final manuscript was carefully revised and approved by the human author, who takes full responsibility for its content.

\bibliographystyle{abbrvnat}
\bibliography{discrete_nadin_periodic_lattices}

@article{BerestyckiHamelNadirashvili2005,
  author  = {Berestycki, Henri and Hamel, Fran{\c c}ois and Nadirashvili, Nikolai},
  title   = {The speed of propagation for {KPP} type problems. {I}. {P}eriodic framework},
  journal = {J. Eur. Math. Soc. (JEMS)},
  volume  = {7},
  number  = {2},
  pages   = {173--213},
  year    = {2005},
  doi     = {10.4171/JEMS/26}
}

@article{BurkardEtAl1998,
  author  = {Burkard, Rainer E. and Deineko, Vladimir G. and van Dal, Ren{\'e}
             and van der Veen, Jack A. A. and Woeginger, Gerhard J.},
  title   = {Well-solvable special cases of the traveling salesman problem:
             a survey},
  journal = {SIAM Rev.},
  volume  = {40},
  number  = {3},
  pages   = {496--546},
  year    = {1998},
  doi     = {10.1137/S0036144596297514}
}

@misc{CasselMannorTennenholtz2020,
  author        = {Cassel, Asaf and Mannor, Shie and Tennenholtz, Guy},
  title         = {The pendulum arrangement: maximizing the escape time of
                   heterogeneous random walks},
  year          = {2020},
  eprint        = {2007.13232},
  archivePrefix = {arXiv},
  primaryClass  = {math.PR},
  note          = {arXiv:2007.13232}
}

@article{CelaDeinekoWoeginger2017,
  author  = {{\c C}ela, Eranda and Deineko, Vladimir G. and Woeginger, Gerhard J.},
  title   = {The multi-stripe travelling salesman problem},
  journal = {Ann. Oper. Res.},
  volume  = {259},
  number  = {1--2},
  pages   = {21--34},
  year    = {2017},
  doi     = {10.1007/s10479-017-2513-4}
}

@article{Cheng2015,
  author  = {Cheng, Cui-Ping},
  title   = {Travelling wave solutions in periodic monostable equations
             on a two-dimensional spatial lattice},
  journal = {IMA J. Appl. Math.},
  volume  = {80},
  number  = {4},
  pages   = {1254--1272},
  year    = {2015},
  doi     = {10.1093/imamat/hxu038}
}

@article{DonskerVaradhan1975,
  author  = {Donsker, Monroe D. and Varadhan, Srinivasa R. S.},
  title   = {On a variational formula for the principal eigenvalue for
             operators with maximum principle},
  journal = {Proc. Natl. Acad. Sci. USA},
  volume  = {72},
  number  = {3},
  pages   = {780--783},
  year    = {1975},
  doi     = {10.1073/pnas.72.3.780}
}

@article{DonskerVaradhan1976,
  author  = {Donsker, Monroe D. and Varadhan, Srinivasa R. S.},
  title   = {On the principal eigenvalue of second-order elliptic differential operators},
  journal = {Comm. Pure Appl. Math.},
  volume  = {29},
  number  = {6},
  pages   = {595--621},
  year    = {1976},
  doi     = {10.1002/cpa.3160290606}
}

@book{EkelandTemam1999,
  author    = {Ekeland, Ivar and T{\'e}mam, Roger},
  title     = {Convex Analysis and Variational Problems},
  series    = {Classics in Applied Mathematics},
  volume    = {28},
  publisher = {Society for Industrial and Applied Mathematics},
  address   = {Philadelphia, PA},
  year      = {1999},
  doi       = {10.1137/1.9781611971088}
}

@article{Fisher1937,
  author  = {Fisher, Ronald A.},
  title   = {The wave of advance of advantageous genes},
  journal = {Ann. Eugenics},
  volume  = {7},
  number  = {4},
  pages   = {355--369},
  year    = {1937},
  doi     = {10.1111/j.1469-1809.1937.tb02153.x}
}

@article{Friedland1981,
  author  = {Friedland, Shmuel},
  title   = {Convex spectral functions},
  journal = {Linear Multilinear Algebra},
  volume  = {9},
  number  = {4},
  pages   = {299--316},
  year    = {1981},
  doi     = {10.1080/03081088108817381}
}

@article{GuoHamel2006,
  author  = {Guo, Jong-Shenq and Hamel, Fran{\c c}ois},
  title   = {Front propagation for discrete periodic monostable equations},
  journal = {Math. Ann.},
  volume  = {335},
  number  = {3},
  pages   = {489--525},
  year    = {2006},
  doi     = {10.1007/s00208-005-0729-0}
}

@article{HatanoNelson1996,
  author  = {Hatano, Naomichi and Nelson, David R.},
  title   = {Localization transitions in non-{H}ermitian quantum mechanics},
  journal = {Phys. Rev. Lett.},
  volume  = {77},
  number  = {3},
  pages   = {570--573},
  year    = {1996},
  doi     = {10.1103/PhysRevLett.77.570}
}

@article{HatanoNelson1997,
  author  = {Hatano, Naomichi and Nelson, David R.},
  title   = {Vortex pinning and non-{H}ermitian quantum mechanics},
  journal = {Phys. Rev. B},
  volume  = {56},
  number  = {14},
  pages   = {8651--8673},
  year    = {1997},
  doi     = {10.1103/PhysRevB.56.8651}
}

@article{HatanoNelson1998,
  author  = {Hatano, Naomichi and Nelson, David R.},
  title   = {Non-{H}ermitian delocalization and eigenfunctions},
  journal = {Phys. Rev. B},
  volume  = {58},
  number  = {13},
  pages   = {8384--8390},
  year    = {1998},
  doi     = {10.1103/PhysRevB.58.8384}
}

@article{HajaiejHanHua2026,
  author  = {Hajaiej, Hichem and Han, Fengwen and Hua, Bobo},
  title   = {Discrete {S}chwarz rearrangement on lattice graphs},
  journal = {J. Lond. Math. Soc. (2)},
  volume  = {113},
  number  = {6},
  pages   = {e70583},
  year    = {2026},
  doi     = {10.1112/jlms.70583}
}

@article{HamelNadirashviliRuss2011,
  author  = {Hamel, Fran{\c c}ois and Nadirashvili, Nikolai and Russ, Emmanuel},
  title   = {Rearrangement inequalities and applications to isoperimetric
             problems for eigenvalues},
  journal = {Ann. of Math. (2)},
  volume  = {174},
  number  = {2},
  pages   = {647--755},
  year    = {2011},
  doi     = {10.4007/annals.2011.174.2.1}
}

@article{Holland1978,
  author  = {Holland, Charles J.},
  title   = {A minimum principle for the principal eigenvalue for second-order linear elliptic equations with natural boundary conditions},
  journal = {Comm. Pure Appl. Math.},
  volume  = {31},
  number  = {4},
  pages   = {509--519},
  year    = {1978},
  doi     = {10.1002/cpa.3160310406}
}

@article{HolmesHolroydRamirez2023,
  author  = {Holmes, Mark and Holroyd, Alexander E. and Ram{\'i}rez, Alejandro},
  title   = {Cyclic products and optimal traps in cyclic birth and death chains},
  journal = {Electron. J. Combin.},
  volume  = {30},
  number  = {2},
  pages   = {Paper No. P2.52},
  year    = {2023},
  doi     = {10.37236/11494}
}

@book{HardyLittlewoodPolya1952,
  author    = {Hardy, Godfrey H. and Littlewood, John E. and P{\'o}lya, George},
  title     = {Inequalities},
  edition   = {2},
  publisher = {Cambridge University Press},
  address   = {Cambridge},
  year      = {1952}
}

@article{KPP1937,
  author  = {Kolmogorov, Andrey N. and Petrovskii, Ivan G. and Piskunov, Nikolai S.},
  title   = {A study of the diffusion equation with increase in the amount of substance, and its application to a biological problem},
  journal = {Bull. Moscow Univ. Math. Mech.},
  volume  = {1},
  number  = {6},
  pages   = {1--25},
  year    = {1937},
  note    = {Russian original}
}

@article{LiangLinMatano2010,
  author  = {Liang, Xing and Lin, Xiaotao and Matano, Hiroshi},
  title   = {A variational problem associated with the minimal speed of travelling waves for spatially periodic reaction-diffusion equations},
  journal = {Trans. Amer. Math. Soc.},
  volume  = {362},
  number  = {11},
  pages   = {5605--5633},
  year    = {2010},
  doi     = {10.1090/S0002-9947-2010-04931-1}
}

@article{LiangMatano2014,
  author  = {Liang, Xing and Matano, Hiroshi},
  title   = {Maximizing the spreading speed of {KPP} fronts in two-dimensional stratified media},
  journal = {Proc. Lond. Math. Soc. (3)},
  volume  = {109},
  number  = {5},
  pages   = {1137--1174},
  year    = {2014},
  doi     = {10.1112/plms/pdu031}
}

@article{LiangZhao2010,
  author  = {Liang, Xing and Zhao, Xiao-Qiang},
  title   = {Spreading speeds and traveling waves for abstract monostable evolution systems},
  journal = {J. Funct. Anal.},
  volume  = {259},
  number  = {4},
  pages   = {857--903},
  year    = {2010},
  doi     = {10.1016/j.jfa.2010.04.018}
}

@article{LiangZhou2020,
  author  = {Liang, Xing and Zhou, Tao},
  title   = {Spreading speeds of {KPP}-type lattice systems in heterogeneous media},
  journal = {Commun. Contemp. Math.},
  volume  = {22},
  number  = {1},
  pages   = {1850083},
  year    = {2020},
  doi     = {10.1142/S0219199718500839}
}

@article{Nadin2010,
  author  = {Nadin, Gr{\'e}goire},
  title   = {The effect of the {S}chwarz rearrangement on the periodic principal eigenvalue of a nonsymmetric operator},
  journal = {SIAM J. Math. Anal.},
  volume  = {41},
  number  = {6},
  pages   = {2388--2406},
  year    = {2010},
  doi     = {10.1137/080743597}
}

@article{Steinerberger2024,
  author  = {Steinerberger, Stefan},
  title   = {Discrete rearrangements and the {P}{\'o}lya--{S}zeg{\H o} inequality on graphs},
  journal = {Studia Math.},
  volume  = {274},
  number  = {3},
  pages   = {269--286},
  year    = {2024},
  doi     = {10.4064/sm230526-15-10}
}

@article{Supnick1957,
  author  = {Supnick, Fred},
  title   = {Extreme {H}amiltonian lines},
  journal = {Ann. of Math. (2)},
  volume  = {66},
  number  = {1},
  pages   = {179--201},
  year    = {1957},
  doi     = {10.2307/1970124}
}

@article{Weinberger2002,
  author  = {Weinberger, Hans F.},
  title   = {On spreading speeds and traveling waves for growth and migration models in a periodic habitat},
  journal = {J. Math. Biol.},
  volume  = {45},
  number  = {6},
  pages   = {511--548},
  year    = {2002},
  doi     = {10.1007/s00285-002-0169-3}
}

\end{document}